\documentclass[12pt,a4paper,dvipsnames]{article}
\usepackage[utf8]{inputenc}
\usepackage{hyperref}

\usepackage{multicol}
\PassOptionsToPackage{dvipsnames}{xcolor}
\usepackage{tikz,xcolor}

\usetikzlibrary{shapes,arrows}
\usetikzlibrary{hobby}
\usetikzlibrary{decorations.markings}

\usepackage{geometry}
\usepackage{amsmath,amsfonts,amsthm, mathrsfs}
\theoremstyle{plain}
\newtheorem{theorem}{Theorem}[section]

\newtheorem{lemma}[theorem]{Lemma}
\newtheorem{proposition}[theorem]{Proposition}

\theoremstyle{definition}

\numberwithin{equation}{section}
\numberwithin{table}{section}

\renewcommand{\P}{\mathbb P}
\newcommand{\TV}[1]{{\lVert #1 \rVert}_{\normalfont
\text{TV}}}
\newcommand{\N}{\mathbb{N}}
\newcommand{\R}{\mathbb{R}}

\title{Limit profile for the ASEP with one open boundary}
\author{Jimmy He\footnote{\textit{Massachusetts Institute of Technology, United States. E-Mail}: \nolinkurl{jimmyhe@mit.edu} } \ and Dominik Schmid\footnote{\textit{University of Bonn, Germany. E-Mail}: \nolinkurl{d.schmid@uni-bonn.de}}}
\date{\today}

\begin{document}

\maketitle
\begin{abstract} We study the speed of convergence to equilibrium for the asymmetric simple exclusion process (ASEP) on a finite interval with one open boundary. We provide sharp estimates on the total-variation distance from equilibrium and verify that the limit profile undergoes a phase transition from a Gaussian to a KPZ profile.
%As a key tool, we interpret the ASEP with one open boundary as a random walk on a Hecke algebra inspired by recent work of Bufetov and Nejjar. 
\end{abstract}

\section{Introduction}

In recent years, understanding the speed of convergence to equilibrium of asymmetric exclusion processes became a central task in the area of mixing times; see for example \cite{BN:CutoffASEP,ES:HighLow,GNS:MixingOpen,LL:CutoffASEP,LL:CutoffWeakly,S:MixingTASEP,SS:TASEPcircle}. We are interested in a more refined description of the distance from the stationary distribution, the limit profile. For asymmetric simple exclusion processes with one open boundary, we show that depending on boundary parameters, the window of convergence changes from a diffusive scaling to a KPZ scaling. \\

More precisely, we consider the asymmetric simple exclusion process (ASEP) on a segment of size $N$ with drift parameter $q \in (0,1)$ and boundary parameters $\alpha,\gamma>0$.
Then the \textbf{ASEP with one open boundary} is a continuous-time Markov chain $(\eta_t)_{t \geq 0}$ with state space $\Omega_{N} := \lbrace 0,1 \rbrace^{N}$ according to the following description. We say that a site $x$ is \textbf{occupied} by a particle if $\eta(x)=1$ and \textbf{vacant} otherwise. 
Each a particle at site $x$ is equipped with independent rate $1$ and rate $q$ Poisson clocks. Whenever the first clock rings, the particle attempts a jump to the right, whenever the second clock rings a jump to the left. However, the jump is performed if and only if the target is a vacant site. In addition, we assign independent rate $\alpha$ and rate $\gamma$ Poisson clocks to at site $1$. Whenever the rate $\alpha$ clock rings and site $1$ is vacant, we place a particle. Similarly, when the rate $\gamma$ clock rings and site $1$ is occupied, we remove the particle from site $1$; see Figure \ref{fig:BDEP} for a visualization. Let us stress that the total number of particles will not be preserved over time. For $\alpha,\gamma>0$, let $\rho\in (0,1)$ with
\begin{equation*}
    \rho=\frac{1}{1+\kappa_+} \qquad \text{where} \qquad
    \kappa_+=\frac{1-q+\gamma-\alpha+\sqrt{(1-q+\gamma-\alpha)^2+4\alpha\gamma}}{2\alpha}
\end{equation*}
denote the \textbf{effective density} created by boundary interaction at site $1$. Moreover, we impose the assumption $\alpha > \gamma$, which automatically is satisfied when $\rho \geq \frac{1}{2}$.\\

 Note that the ASEP with one open boundary has a unique stationary distribution $\mu=\mu^{N,q,\alpha,\gamma}$ on the segment $[N] := \{1,\dots,N\}$. Our main goal is to study the distance from $\mu$ in total variation over time. We let $\TV{ \ \cdot \ }$ denote the \textbf{total-variation distance}, i.e.\ for two probability measures $\nu$ and $\nu^{\prime}$ on $\Omega_N$, we set
\begin{equation}\label{def:TVDistance}
\TV{ \nu - \nu^{\prime} } := \frac{1}{2}\sum_{x \in \Omega_N} |\nu(x)-\nu^{\prime}(x)| = \max_{A \subseteq \Omega_N} \left(\nu(A)-\nu^{\prime}(A)\right) \ .
\end{equation} Consider now the distance
\begin{equation}\label{def:MixingTime}
d_N(t) := \max_{\eta \in \Omega_{N}} \TV{\P\left( \eta_t \in \cdot \ \right | \eta_0 = \eta) - \mu} 
\end{equation} for all $t \geq 0$.
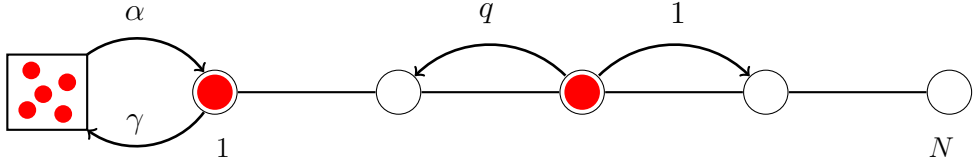
\begin{figure} \label{fig:BDEP}
\centering
\begin{tikzpicture}[scale=1.05]

\def\spiral[#1](#2)(#3:#4:#5){% \spiral[draw options](placement)(end angle:revolutions:final radius)
\pgfmathsetmacro{\domain}{pi*#3/180+#4*2*pi}
\draw [#1,
       shift={(#2)},
       domain=0:\domain,
       variable=\t,
       smooth,
       samples=int(\domain/0.08)] plot ({\t r}: {#5*\t/\domain})
}

\def\particles(#1)(#2){

  \draw[black,thick](-3.9+#1,0.55-0.075+#2) -- (-4.9+#1,0.55-0.075+#2) -- (-4.9+#1,-0.4-0.075+#2) -- (-3.9+#1,-0.4-0.075+#2) -- (-3.9+#1,0.55-0.075+#2);
  
  	\node[shape=circle,scale=0.6,fill=red] (Y1) at (-4.15+#1,0.2-0.075+#2) {};
  	\node[shape=circle,scale=0.6,fill=red] (Y2) at (-4.6+#1,0.35-0.075+#2) {};
  	\node[shape=circle,scale=0.6,fill=red] (Y3) at (-4.2+#1,-0.2-0.075+#2) {};
   	\node[shape=circle,scale=0.6,fill=red] (Y4) at (-4.45+#1,0.05-0.075+#2) {};
  	\node[shape=circle,scale=0.6,fill=red] (Y5) at (-4.65+#1,-0.15-0.075+#2) {}; }

  \def\annhil(#1)(#2){	  \spiral[black,thick](9.0+#1,0.09+#2)(0:3:0.42);
  \draw[black,thick](8.5+#1,0.55+#2) -- (9.5+#1,0.55+#2) -- (9.5+#1,-0.4+#2) -- (8.5+#1,-0.4+#2) -- (8.5+#1,0.55+#2); }

	\node[shape=circle,scale=1.5,draw] (B) at (2.3,0){} ;
	\node[shape=circle,scale=1.5,draw] (C) at (4.6,0) {};
	\node[shape=circle,scale=1.2,fill=red] (CB) at (2.3,0) {};
    \node[shape=circle,scale=1.5,draw] (A) at (0,0){} ;
 %   \node[shape=circle,scale=1.5,fill=red] (AB) at (0,0){} ;    
 	\node[shape=circle,scale=1.5,draw] (D) at (6.9,0){} ; 
% 		\node[shape=circle,scale=1.5,fill=red] (CD) at (6.9,0) {};
 	 	\node[shape=circle,scale=1.5,draw] (Z) at (-2.3,0){} ;
	\node[shape=circle,scale=1.2,fill=red] (YZ) at (-2.3,0) {};
   \node[line width=0pt,shape=circle,scale=1.6] (B2) at (2.3,0){};
	\node[line width=0pt,shape=circle,scale=2.5] (D2) at (6.9,0){};
		\node[line width=0pt,shape=circle,scale=2.5] (Z2) at (-2.3,0){};
		
	\node[line width=0pt,shape=circle,scale=2.5] (X10) at (6.8,0){};
		\node[line width=0pt,shape=circle,scale=2.5] (X11) at (-2.2,0){};	
			\node[line width=0pt,shape=circle,scale=2.5] (X12) at (8.4,0){};
		\node[line width=0pt,shape=circle,scale=2.5] (X13) at (-3.8,0){};

		\draw[thick] (Z) to (A);	
	\draw[thick] (A) to (B);	
		\draw[thick] (B) to (C);	
  \draw[thick] (C) to (D);

\particles(0)(0);
%\particles(6.9+4.9+1.6)(0);

%\annhil(0)(0.7);
%\annhil(-13.4)(-0.7);

\draw [->,line width=1pt]  (B2) to [bend right,in=135,out=45] (C);
  
  % \draw [->,line width=1pt] (B2) to [bend right,in=-135,out=-45] (C);
   \draw [->,line width=1pt] (B2) to [bend right,in=-135,out=-45] (A);
 %     \draw [->,line width=1pt] (A) to [bend right,in=-135,out=-45] (Z2);
    \node (text1) at (3.5,1){$1$} ;    
	\node (text2) at (1.1,1){$q$} ;   
	\node (text3) at (-2.3-1,1){$\alpha$}; 
	%\node (text4) at (6.9+1,1){$\beta$};   
	\node (text5) at (-2.3-1,-0.4){$\gamma$}; 
	%\node (text6) at (6.9+1,-0.4){$\delta$};   

    \node[scale=0.9] (text1) at (-2.2,-0.7){$1$} ;    
    \node[scale=0.9] (text1) at (6.8,-0.7){$N$} ;   
  	
  \draw [->,line width=1pt] (-3.9,0.475) to [bend right,in=135,out=45] (Z);
   \draw [->,line width=1pt] (Z) to [bend right,in=135,out=45] (-3.9,-0.475); 
   %\draw [->,line width=1pt] (6.9+1.6,-0.475) to [bend right,in=135,out=45] (D);	
   %\draw [->,line width=1pt] (D) to [bend right,in=135,out=45] (6.9+1.6,0.475);	

	\end{tikzpicture}	
\caption{The ASEP with one open boundary for parameters $q,\alpha$ and $\gamma$.}
 \end{figure}Depending on the effective density $\rho$, we define the function
\begin{equation}\label{def:TimeFunction}
g_{\rho}(c) = g_{\rho}^{N}(c) := \frac{1}{1-q} \begin{cases}
4N  + c 2^{-2/3}N^{1/3} \qquad & \text{if } \rho > \frac{1}{2} \\
4N  + c 2^{-2/3}N^{1/3} \qquad &\text{if } \rho = \frac{1}{2} \\
\frac{1}{\mu}N  + c \mu^{-1/2}\sigma^{-1} N^{1/2}\qquad &\text{if } \rho < \frac{1}{2} , \\
\end{cases}
\end{equation} where we set $\mu=\rho(1-\rho)$ and $\sigma^{2}=\rho(1-\rho)(1-2\rho)$.
In the following, we let $F_{\textup{GOE}},F_{\textup{GSE}}$ and $\Phi$ denote the cumulative distribution function of the Tracy-Widom-$GOE$, Tracy-Widom-$GSE$, and Gaussian distribution, respectively. 

\begin{theorem}\label{thm:Main} Let $q<1$ and $\rho > 0$, as well as $\alpha>\gamma$. Then for all $c \in \R$
\begin{equation}\label{eq:Main}
\lim_{N \rightarrow \infty}d_N( g_{\rho}(c)) = 1-F_{\rho}(c) := \begin{cases}
 1-F_{\textup{GSE}}(c) \qquad &\text{if } \rho > \frac{1}{2} \\
 1-F_{\textup{GOE}}(c) \qquad &\text{if } \rho = \frac{1}{2} \\
 1-\Phi(c) \qquad &\text{if } \rho < \frac{1}{2}. \\
\end{cases} 
\end{equation}
\end{theorem}
Let us remark that a similar statement as in Theorem \ref{thm:Main} holds for $q > 1$ using the particle-hole symmetry, i.e.\ $d_N(t)$ remains unchanged when swapping $1$ and $q$ for the jump rates as well as swapping the roles of $\alpha$ and $\gamma$. Also, fixing one of the rates to be $1$ causes no loss of generality, as we may always do so by rescaling time. Finally, our techniques should also be able to allow $\rho$ to vary with $N$ to obtain distributions interpolating between the $GSE$ and $GOE$ distributions.
For $\varepsilon \in (0,1)$, let 
\begin{equation}
t^{N}_{\textup{mix}}(\varepsilon) := \inf\left\{ t \geq 0 \, \colon \, d_N(t) \leq  \varepsilon \right\}
\end{equation} be the $\varepsilon$-mixing time of the ASEP with one open boundary.  As previously shown by Gantert et al.\ in \cite{GNS:MixingOpen},  for  $q<1$,  $t^{N}_{\textup{mix}}(\varepsilon)$ is of order $4N$ for $\rho\geq 1/2$, and order $N\rho^{-1}(1-\rho)^{-1}$ for  $\rho <\frac{1}{2}$. In particular, the leading order does not depend on~$\varepsilon$. This is called the \textbf{cutoff phenomenon}. Here, we give a more precise description of the distance to equilibrium, depending on the parameter $\rho$. Let us stress that the behavior for two open boundaries is fundamentally different, for example the TASEP with two open boundaries, where $q=0$ and particles can enter at the site $1$ and exit at site $N$ of the segment, is known to exhibit several different mixing time regimes, some of them without cutoff. 

\subsection{Related work}

Mixing times are a central object in the study of Markov chains, and we refer to \cite{LPW:markov-mixing} for a general introduction.
In recent years, the cutoff phenomenon for exclusion processes has been intensively studied. We refer to a seminal series by Lacoin \cite{L:CutoffCircle,L:CutoffSEP,L:CycleDiffusiveWindow} for the symmetric simple exclusion process on the line and circle, and more recently results by Gantert et al., Salez, and Tran for the symmetric simple exclusion process with open boundaries \cite{GNS:MixingOpen,S:CutoffExclusion,T:SSEP}. For asymmetric simple exclusion processes on a closed segment, mixing times were first studied in \cite{BBHM:MixingBias}, while cutoff was shown in \cite{LL:CutoffASEP}; see also \cite{LL:CutoffWeakly} for cutoff for the weakly asymmetric simple exclusion process and \cite{ES:HighLow,S:MixingTASEP,SS:TASEPcircle} for mixing times of the totally asymmetric simple exclusion process with two open boundaries. 
In \cite{BN:CutoffASEP}, the limit profile of the ASEP on the segment is established by Bufetov and Nejjar. We will crucially rely on their approach by interpreting the ASEP as a random walk on a Hecke algebra; see also \cite{BB:ColorPosition,B:Hecke, K:Cox}. Beside the results in \cite{BN:CutoffASEP}, there are also results on limit profiles for multispecies exclusion processes \cite{Z:Metropolis} and related results on hitting times in the totally asymmetric case, sometimes referred to as the oriented swap process \cite{BGR:OSP, H:shift}. Limit profiles were also studied for symmetric exclusion and interchange processes \cite{L:CutoffCircle,NO:LimitInterchange}, as well as for several card shuffles and, more generally, reversible Markov chains \cite{N:LimitComparison,NO:LimitReversible,T:LimitProfile}. 
Finally, let us remark that as a key tool, we use current estimates for the half space ASEP established very recently by the first author in \cite{H:Boundary}. 
Such results were previously only available when  $\rho=\frac{1}{2}$ by Barraquand et al.\ in \cite{BBCW:HalfspaceASEP}, and the totally asymmetric simple exclusion process by a connection to Pfaffian Schur processes in \cite{BBCS:Halfspace}.

\subsection{Outline of the paper}

In Section \ref{sec:ASEPInteracting}, we study the ASEP with one open boundary as an interacting particle system. We discuss the canonical coupling and describe its extension to the multi-species ASEP as well as the half space ASEP. Moreover, we recall some recent results on the current of the half space ASEP and express the distance from equilibrium in terms of hitting times. This allows us to establish the lower bound in Theorem \ref{thm:Main}. In Section \ref{sec:ASEPHecke}, we view the multi-species ASEP with one open boundary as a random walk on a type $BC$ Hecke algebra, and discuss the existence of a Mallows element. The upper bound on the distance from stationarity in Theorem \ref{thm:Main} is established in Section \ref{sec:LimitProfileUpperBound} following the strategy introduced by Bufetov and Nejjar for the ASEP on the segment. 

\section{The ASEP as an interacting particle system}\label{sec:ASEPInteracting}

In the following, we collect several facts about the ASEP with one open boundary from an interacting particle system perspective. This includes the canonical coupling to multi-species exclusion processes and the half-line ASEP, current theorems, and hitting times. 

\subsection{Canonical coupling for the ASEP with one open boundary}\label{sec:CanonicalCoupling}

We start by introducing the \textbf{canonical coupling} for the ASEP with one open boundary. Let $\eta,\eta^{\prime}$ be two initial configurations for two ASEPs $(\eta_t)_{t \geq 0}$ and $(\eta^{\prime}_t)_{t \geq 0}$ on $\Omega_N$. We couple the two processes in such a way that we assign independent rate $1$ and rate $q$ Poisson clocks to all edges. Whenever the rate $1$ clock rings at time $t$ for an edge $\{x,x+1\}$ with $x\in [N-1]$, and $\eta_{t_-}(x)=1-\eta_{t_-}(x+1)=1$, then move the particle from $x$ to $x+1$. Otherwise, leave the configuration unchanged. Similarly, we move the particle from $x$ to $x+1$ in $(\eta^{\prime}_t)_{t \geq 0}$ whenever $\eta^{\prime}_{t_-}(x)=1-\eta^{\prime}_{t_-}(x+1)=1$, and leave the configuration unchanged otherwise. For the rate $q$ proceed similarly. In addition, we use the same rate $\alpha$ and rate $\gamma$ Poisson clocks for $(\eta_t)_{t \geq 0}$ and $(\eta^{\prime}_t)_{t \geq 0}$ to determine when a particle attempts to enter, respectively to exit, the segment. 
%Note that the canonical coupling preserves the partial order 
%
%; see Lemma ... in \cite{GNS:MixingOpen}.  
Next, we discuss two straightforward extensions of the canonical coupling. First, we consider the \textbf{multi-species  extension}. 
Let $\bar{\Omega}_N := \{1,2,3,\infty\}^{N}$, and consider the partial ordering introduced by component-wise ordering on $\bar{\Omega}_N$ with respect to the total ordering $1  >_p  2  >_p 3 >_p \infty$. We call a site $x$ in $\zeta \in \bar{\Omega}_N$ occupied by a \textbf{first} (respectively \textbf{second} or \textbf{third}) \textbf{class particle} if $\eta(x)=1$ (respectively $\eta(x)=2$ or $\eta(x)=3$), and a \textbf{hole} if $\eta(x)=\infty$. For the canonical coupling of the multi-species ASEP, whenever a rate $1$ clock rings along an edge $e$, we sort the configuration along $e$ in increasing order according to $>_p$, and for rate $q$ clocks in decreasing order. If the rate $\alpha$ clock at site $1$ rings at time $t$, and $\eta_{t_{-}}(1)=\infty$ ($\eta_{t_{-}}(1)=3$), then place a first class particle (second class particle) at site $1$. Similarly, for the rate $\gamma$ clock and $\eta_{t_{-}}(1)=1$ ($\eta_{t_{-}}(1)=2$), we place a hole (third class particle) at site $1$. Let us remark that we obtain in the same way a \textbf{fully colored ASEP} by suitably assigning to all sites different colors; see Section \ref{sec:ASEPHecke} for a formal definition.  
We make the following observation on the multi-species ASEP.

\begin{lemma}\label{lem:CouplingMultiSpecies}
Consider the projection $(\bar{\zeta}_t)_{t \geq 0}$ of the multi-species $(\zeta_t)_{t \geq 0}$ ASEP where we map first and second class particles to particles, and third class particles and holes to empty sites. Then $(\bar{\zeta}_t)_{t \geq 0}$ has the law of an ASEP with one open boundary.
\end{lemma}
\begin{proof}
This is follows immediately by verifying the marginal transition rates.
\end{proof}
%\begin{remark}
%Let us stress that this multi-species coupling is different from the disagreement coupling constructed in \cite{GNS:MixingOpen}. In particular, note that in the above coupling the total number of second class particles in the segment is preserved over time. 
%\end{remark}
For the second extension, consider the \textbf{half space ASEP} $(\xi_t)_{t \geq 0}$. Here, the state space is given by $\Omega := \{ 0,1 \}^{\N}$, particles enter at site $1$ according to rate $\alpha$ Poisson clocks, respectively exit at site $1$ at rate $\gamma$ Poisson clocks. Inside the half open line, particles move to the right (left) according to rate $1$ (rate $q$) Poisson clocks under the exclusion rule. Note that we can couple an ASEP with one open boundary $(\eta_t)_{t \geq 0}$ and a half space ASEP $(\xi_t)_{t \geq 0}$ by using the same clocks on sites $1$ up to $N$, including the clocks for entering and exiting of particles, as well as independent clocks for $(\xi_t)_{t \geq 0}$ at all sites larger than $N$. Again, we refer to this as the canonical coupling. For any $\eta \in \Omega_N$ and $\zeta \in \Omega_M$, allowing $M=\infty$ with $\Omega=\Omega_\infty$ if $\zeta(x)\neq 1$ for finitely many $x\in \N$, we can define a partial ordering by
\begin{equation}
\zeta \succeq \eta \qquad \Leftrightarrow \qquad \sum_{i=x}^{M} (1-\zeta(i)) \geq \sum_{i=x}^{N} (1-\eta(i))  \ \text{ for all } x \in \N . 
\end{equation} 
In particular, note that this ordering allows us to compare exclusion processes on different state spaces. For a configuration $\eta$, let $\mathcal{R}(\eta) = \sup\{ x \in \N \, \colon \, \eta(x)=0 \}$ denote the location of the right-most empty site in $\eta$. A key tool is that the canonical coupling preserves this order.

\begin{lemma}\label{lem:DominationLeftmostParticle} Let $\eta \in \Omega_N$ and $\zeta \in \Omega_M$ such that $\eta \preceq \zeta$. Then under the canonical coupling $\mathbf{P}$, 
\begin{equation}\label{eq:GeneralDomination}
\mathbf{P}\left(  \eta_t \preceq \zeta_t \text{ for all } t \geq 0  \mid \eta_0=\eta \text{ and } \zeta_0=\zeta \right) = 1
\end{equation}
In particular, the condition $\mathcal{R}(\eta)\leq \mathcal{R}(\zeta)< \infty$ is preserved for all time, or in other words, 
\begin{equation}\label{eq:RightmostDomination}
\mathbf{P}\left(  \mathcal{R}(\eta_t) \leq \mathcal{R}(\zeta_t) \text{ for all } t \geq 0  \mid \eta_0=\eta \text{ and } \zeta_0=\zeta \right) = 1 .
\end{equation}
\end{lemma}
\begin{proof} For two configurations $\eta,\zeta \in \Omega_N$, the first statement is Lemma 2.3 in \cite{GNS:MixingOpen}. Now let $N \leq M$, and note that we can extend $(\eta_t)_{t \geq 0}$ to a configuration $(\eta^{\prime}_{t})_{t \geq 0}$ in $\Omega_M$ by $\eta^{\prime}_t=(\eta_t,1,\dots,1)$ for all $t \geq 0$. Note that $(\eta^{\prime}_t)_{t \geq 0}$ has the law of an ASEP with open boundary on $\Omega_M$ where we suppress all jumps attempts between $N$ and $N+1$. We see that this must also preserve the ordering, as the suppressed jumps can only increase the configuration under this ordering. The second statement follows from the first.
%Note that when the right-most particle in $(\zeta_t)_{t \geq 0}$ moves to the left at time $s$, and $\mathcal{R}(\zeta_{s_{-}})=\mathcal{R}(\eta_{s_{-}})=r$ for some $r\in [N]$, then we must have $\zeta_{s_{-}}(r-1)=\eta_{s_{-}}(r-1)=0$, and hence $\mathcal{R}(\eta_s)=\mathcal{R}(\zeta_s)=r-1$. A similar argument holds for the remaining particles, and he claim follows now by induction over the jump times of $(\mathcal{R}(\zeta_{t}))_{t \geq 0}$.
\end{proof}

\subsection{Current theorems for the ASEP on the half-space}\label{sec:CurrentTheorems}

Suppose that we start both dynamics from the empty configuration, denoted by $\mathbf{0}$. Then let $(\mathcal{J}^{N}_t)_{t \geq 0}$ and $(\mathcal{J}^{\N}_t)_{t \geq 0}$ denote the \textbf{current} in $(\eta_{t})_{t \geq 0}$ and $(\xi_t)_{t \geq 0}$ respectively, where 
\begin{equation}
 \mathcal{J}^{N}_t := \lVert \eta_t \rVert_1 \qquad  \text{and} \qquad \mathcal{J}^{\N}_t := \lVert \xi_t \rVert_1 . 
\end{equation}
In words, the current counts the number of particles which have entered the segment by time $t$ minus the number of particles which have exited the segment by time $t$. We have that the current on the half-space stochastically dominates the current on the segment.

\begin{lemma}\label{lem:CurrentComparison}
Suppose that we start with two configurations $\eta \in \{0,1\}^{N}$ and $\xi \in \{0,1\}^{\N}$ such that $\eta(x)=\xi(x)$ for all $x \in [N]$, and $\xi(y)=0$ for all $y>N$. Then under the canonical coupling $\mathbf{P}$, 
\begin{equation}
\mathbf{P} \big( \mathcal{J}^{N}_t \leq  \mathcal{J}^{\N}_t \text{ for all  } t\geq 0 \, \big| \,  \eta_0=\eta \text{ and } \xi_0=\xi \big) = 1
\end{equation}
\end{lemma}
\begin{proof}
This is immediate from the canonical coupling and Lemma \ref{lem:DominationLeftmostParticle}.
\end{proof}

The following precise bounds for the current of the half-space ASEP where recently shown by the first author in \cite{H:Boundary}. The  case $\alpha=\frac{1}{2}$ and $\gamma=\frac{q}{2}$ was previously shown in~\cite{BBCW:HalfspaceASEP}.

\begin{theorem}\label{thm:CurrentEstimate}
Consider the half space ASEP started from the empty configuration, and assume that $q<1$ and $\alpha>\gamma>0$. Let $\rho$ denote the effective density of particles at $0$. Then for all $\rho \geq \frac{1}{2}$, and for $s \in \R$,
\begin{equation}
 \lim_{t \rightarrow \infty}\P\left( \mathcal{J}_{t/(1-q)}^{\N}-\frac{t}{4} \geq - 2^{-4/3} t^{1/3} s \right) = F_{\rho}(s) , 
\end{equation} 
where we recall that $F_{\rho}$ is the Tracy--Widom GSE distribution when $\rho>\frac{1}{2}$ the Tracy--Widom GOE distribution when $\rho=\frac{1}{2}$. 
Similarly, for  $\rho < \frac{1}{2}$, we have that
\begin{equation}
 \lim_{t \rightarrow \infty}\P\left( \mathcal{J}_{t/(1-q)}^{\N}- \mu t \geq - \sigma t^{1/2} s \right) = \Phi(s) , 
\end{equation} where $\Phi$ is the distribution of a standard Gaussian, and we recall $\mu,\sigma$ from \eqref{def:TimeFunction}.
\end{theorem}

\subsection{Hitting times for the ASEP with one open boundary}

We reduce the limit profile bound in Theorem \ref{thm:Main} to a bound on leaving $\mathcal{A}_N$, where
\begin{equation}\label{def:HittingSet}
\mathcal{A}_N := \left\{ \eta  \in \Omega_M \, \colon \, \mathcal{R}(\eta) \geq \log^{1/16}(N) \right\} , 
\end{equation} where $M$ will be clear from the context, and where $\mathcal{R}(\eta)$ denotes the right-most empty site in the configuration $\eta$. A straightforward computation for the invariant measure $\mu$ -- see Lemma \ref{lem:MallowsBound} -- shows that 
\begin{equation}
\lim_{N \rightarrow \infty}\mu(\mathcal{A}_N) = 0 
\end{equation} whenever $q<1$ and $\alpha>\gamma>0$. We have the following relation to the limit profile.
\begin{lemma}\label{lem:HittingTimesLemma}
Consider the initial configuration where we start from the all empty configuration. Suppose there exists some constant $c_1$ and a sequence $(T_N)_{N \in \N}$ such that
\begin{equation}
\limsup_{N \rightarrow \infty}{\P}_N\left( \eta_{T_N} \notin \mathcal{A}_N \right) \leq c_1 ,
\end{equation}
Then the distance from stationarity  satisfies
\begin{equation}\label{eq:LowerBoundHitting}
\liminf_{N \rightarrow \infty} d_N(T_N) \geq 1-  c_1 . 
\end{equation} 
Conversely, suppose that we find some constant $c_2$ and a sequence $(T^{\prime}_N)_{N \in \N}$ such that
\begin{equation}
\liminf_{N \rightarrow \infty}{\P}_N\left( \eta_{t} \notin \mathcal{A}_N \text{ for some } t \in [0,T^{\prime}_N] \right) \geq c_2 . 
\end{equation} Then the distance from stationarity satisfies
\begin{equation}\label{eq:UpperBoundHitting}
\limsup_{N \rightarrow \infty} d_N(T^{\prime}_N+N^{1/4}) \leq 1-c_2 . 
\end{equation}
\end{lemma}
\begin{proof} The first statement is immediate from the definition of $d_N(t)$ and the total-variation distance. For the second statement, let $\tau$ be the first time $t$ where $\eta_{t} \notin \mathcal{A}_N$. Using the strong Markov property of the ASEP with one open boundary, Lemma~5.2 in \cite{GNS:MixingOpen} to bound the expected hitting time of $\mathbf{1}$, and Markov's inequality, we get
\begin{equation*}
\limsup_{N\to\infty}\max_{ \zeta \notin \mathcal{A}_N } {\P}_{N} \big( \eta_{s} = \mathbf{1} \text{ for some } s\leq N^{1/4}  \, \big| \, \eta_0=\zeta \big)=1. 
\end{equation*}
Using again the strong Markov property of the ASEP with one open boundary, we conclude by the standard fact that
\begin{equation}\label{eq:HitFromExtreme}
d_{N}(t) \leq {\P}_N\left( \eta_{s}= \mathbf{1} \text{ for some } s\leq t \, | \, \eta_0=\mathbf{0} \right);
\end{equation} see for example Corollary 2.5 in \cite{GNS:MixingOpen}. 
\end{proof}
We have now all tools to show the lower bound on the limit profile.

\begin{proof}[Proof of the lower bound in Theorem \ref{thm:Main}] Let $\rho \geq \frac{1}{2}$ as  $\rho < \frac{1}{2}$ is similar.  
We consider the empty initial configuration. Since by Lemma~\ref{lem:CurrentComparison}, for all $t\geq 0$, the number of particles in the segment at time $t$ is stochastically dominated by the current of the half space ASEP, Theorem \ref{thm:CurrentEstimate} yields that at time $t=g_{\rho}(c)$, and all $\rho \geq \frac{1}{2}$, $c\in \R$
\begin{equation}
\liminf_{N \rightarrow \infty}{\P}_N( \lVert\eta_t\rVert_1 < N - \log^{1/16}(N)  ) \geq 1- F_{\rho}(c), 
\end{equation} where recall the notation $F_{\rho}(c)$ from \eqref{eq:Main}.  Using \eqref{eq:LowerBoundHitting} in Lemma \ref{lem:HittingTimesLemma}, we conclude the lower bound in Theorem \ref{thm:Main}. 
\end{proof}%NOTE: I removed the varepsilon, as far as I can tell it's not needed.

\section{The ASEP as a random walk on a Hecke algebra}\label{sec:ASEPHecke}

We describe in the following the evolution of the multi-species ASEP with one open boundary as a random walk on a Hecke algebra. This connection was previously noted in \cite{B:Hecke}, but we give a more detailed explanation. For background on Coxeter groups and Hecke algebras, see \cite{BB:Coxeter,H:Coxeter}.

\subsection{The type $\mathbf{BC}$ Hecke algebra}
We let $B_n$ denote the hyperoctahedral group, which is the group of signed permutations on $n$ letters (which we will just call permutations). That is, $B_n$ is the set of permutations $\pi$ of $\{-n,\dotsc,-1,1,\dotsc, n\}$ such that $\pi(-i)=-\pi(i)$. We will view this as a Coxeter group, with generators $s_0=(-1,1)$ and $s_k=(k,k+1)$ for $0<k<n$. We will want to write elements in one line notation, which means we place a number $i$ at position $j$ (for $i,j\in \pm [1,n]$) if $\pi(j)=i$.
We define the \textbf{length} $l(\pi)$ to be the minimum number of generators needed to write $\pi$, and we call any such representation a \textbf{reduced word}. We may write $l(\pi)=l_0(\pi)+l_1(\pi)$, where $l_0$ counts the number of $s_0$'s and $l_1$ counts the number of $s_k$'s for $k>0$ used in any decomposition $\pi=s_{k_1}\dotsm s_{k_n}$ with $n$ minimal, and in particular, this is well defined.

Multiplication by $s_k$ on the left or right changes $l(\pi)$ by exactly $1$, on the left the length increases or decreases depending on whether the numbers $k$ and $k+1$ (or $1$ and $-1$ if $k=0$) are increasing, and on the right the length increases or decreases depending on whether the numbers at positions $k$ and $k+1$ (or $1$ and $-1$ if $k=0$) are increasing.

Let $q,r$ be formal parameters, and for convenience we let $q_i=q$ if $i>0$, and $q_0=r$. Later on, we will set $r=\gamma/\alpha$ to match the ASEP with open boundary. We define the $\mathbf{BC}$ \textbf{Hecke algebra} to be the associative algebra $\mathcal{H}=\mathcal{H}_{q,r}(B_n)$ over $\mathbf{C}(q,r)$ defined by taking as a basis basis $T_w$ for $w\in B_n$, with relations
\begin{alignat*}{2}
T_uT_w&=T_{uw}&&\qquad \text{ if }l(uw)=l(u)+l(w),
\\(T_i+q_i)(T_i-1)&=0,&&
\end{alignat*}
where $T_i=T_{s_i}$, with $s_i$ the Coxeter generators for $B_n$ defined above. We have $T_{id}=1$. The above relations imply that
\begin{align}
    \label{eq: R_k right} T_\pi T_k&=\begin{cases}T_{\pi s_k}&\text{ if }l(\pi s_k)>l(\pi),
    \\q_kT_{\pi s_k}+(1-q_k)T_{\pi}&\text{ if }l(\pi s_k)<l(\pi).
    \end{cases}
    \\\label{eq: R_k left} T_k T_\pi&=\begin{cases}T_{s_k\pi}&\text{ if }l(s_k \pi)>l(\pi),
    \\q_kT_{s_k\pi}+(1-q_k)T_{\pi}&\text{ if }l(s_k\pi)<l(\pi).
    \end{cases}
\end{align}
We say that an element of $\mathcal{H}$ is a \textbf{probability distribution} if the coefficients in the $T_{\pi}$ basis are non-negative and sum to $1$.

Given $[a,b]=\{a,a+1,\dotsc, b\}$, we let $B_{[a,b]}$ be the parabolic subgroup of $B_n$ generated by $s_i$ for $i\in [a,b-1]$, and let $\mathcal{H}_{[a,b]}$ be the parabolic subalgebra generated by $T_{s_i}$ for $i\in [a,b-1]$. Note that $B_{[0,b]}\cong B_b$ and $B_{[a,b]}\cong S_{b-a}$ if $a>0$, where $S_{b-a}$ denotes the symmetric group. Note that $B_{[a,b]}$ always acts on $\pm [a,b]$ if $a>0$, and on $\pm [1,b]$ if $a=0$. We will also write $B_{\pm [a,b]}=B_{[a,b]}$ and $\mathcal{H}_{\pm[a,b]}=\mathcal{H}_{[a,b]}$ for convenience.

A key property is that the Hecke algebra $\mathcal{H}$ has an anti-involution $\iota$, which sends $\iota(T_w)=\iota(T_{w^{-1}})$ for all $w\in B_n$.

\subsection{Hecke algebras and the ASEP}
We now explain how the ASEP with one open boundary can be viewed as a random walk on a Hecke algebra. We specialize $q$ to match the $q$ parameter in the ASEP and set $r=\gamma/\alpha$. 

We can view elements in $B_n$ and its parabolic quotients $B_n/(B_{[0,k]}\times B_{[k+1,n]})$ as configurations in a particle system on a finite half open line segment. In particular, we view $\pi$ as a configuration with a particle of color $\pi(i)$ at position $i$, and an element $\pi\in B_n/(B_{[1,k]}\times B_{[k+1,n]})$ (for any choice of representative) as a configuration where we take the configuration corresponding to $\pi$, and identify the colors $[k+1,n]$ as $\infty$, $[1,k]$ as $1$, and $-[1,k]$ as $2$, and $[-n,-k-1]$ as $3$. Finally, while we keep track of the colors at negative positions, it is clear that by symmetry, we can forget about this information, and only keep track of what occurs at positive positions.

We now wish to introduce dynamics. We place rate $1$ Poisson clocks at each edge $(k,k+1)$ (associated to $s_k$), and a rate $\alpha$ Poisson clock at the edge $(-1,1)$ (associated to $s_0$). For all $t\geq 0$, we define a random element $W_t$ of $\mathcal{H}$ by starting with $W_0=T_{id}$, and multiplying by $T_k$ on the left every time the clock associated to $s_k$ rings. More generally, if we wish to start from an arbitrary initial state $\pi$, we can simply consider the element $W_t T_\pi $, or even more generally any element of $\mathcal{H}$ which is a probability distribution.

\begin{lemma}
The $W_t$ are a probability distribution for all $t\geq 0$, i.e. the coefficients in the $T_\pi$ basis are non-negative, and sum to $1$.
\end{lemma}
\begin{proof}
By linearity and conditioning on the Poisson clocks, it suffices to show this for any deterministic product $T_{k_1}\dotsm T_{k_n}$. This can be done by induction on $n$. In particular, the multiplication rules \eqref{eq: R_k right} and \eqref{eq: R_k left} ensure that no negative coefficients appear, and that the sums of the coefficients are preserved.
\end{proof}

Given an element $W\in \mathcal{H}$ which is a probability distribution, we obtain a distribution on $B_n$ by considering the coefficients in the $T_\pi$ basis as probabilities. The key fact connecting the ASEP with Hecke algebras is that $W_t$ gives the distribution of the ASEP at time $t$.
More precisely, it gives rise to a \textbf{fully colored ASEP} with the following description. For a permutation $\pi$, its colors are given as elements of $\pm[1,N]$. For each edge $e$, we assign independent rate $1$ and rate $q$ Poisson clocks. Whenever the rate $1$ clock rings, we sort the colors along the respective edge in decreasing order, and for the rate $q$ clock in decreasing order. 
At the boundary site $1$, a color $i$ gets replaced with $-i$ at rate $\alpha$ if $i \geq 1$, and rate $\gamma$, otherwise.

\begin{lemma}
Let $W_0$ be a probability distribution in $\mathcal{H}$. Then $W_t$ gives the distribution of the fully colored ASEP with one open boundary at time $t$, started from the random initial configuration $W_0$.
\end{lemma}
\begin{proof}

This is seen after noting that an equivalent description of the half space colored ASEP is to have a single rate $1$ Poisson clock for each edge, and swap with probability either $1$ or $q$, depending on the order of the numbers at the two sites adjacent to that edge. At the boundary, we instead have a clock of rate $\alpha$, and swap with probabilities $1$ or $r=\gamma/\alpha$. We can then couple the Poisson clocks in the ASEP with the ones defining the random walk on the Hecke algebra. The multiplication rule \eqref{eq: R_k right} exactly encodes the probabilities of a swap occurring at a particular edge, depending on the ordering of the numbers at that edge.
\end{proof}
We can obtain our original system of interest, the half open ASEP with just particles and holes, by projecting to the parabolic quotient $B_n/B_{[1,n]}$, which has the effect of labeling all positive numbers as particles and all negative numbers as holes. A similar statement holds when projecting to the multi-species extension for first, second and third class particles. 

\begin{lemma}
Consider the projection of the fully colored ASEP where we map colors $[-N,-K]$ to $1$, colors $[-K+1,-1]$ to $2$, colors $[1,K-1]$ to $3$, and colors $[K,N]$ to $\infty$, for some fixed $K$.
 Then this projection has the same law as the multi-species extension of the ASEP with one open boundary on the interval $[1,N]$.    
\end{lemma}
\begin{proof}
    This follows immediately from verifying the marginal transition rates.
\end{proof}
We note the following useful property.
\begin{lemma}
\label{lem: W_t inv}
We have $\iota(W_t)$ has the same distribution as $W_t$.
\end{lemma}
\begin{proof}
This follows from the fact that after conditioning on the number of times each Poisson clock rings, $W_t=T_{k_1}\dotsm T_{k_m}$ is uniformly distributed among possible orderings for the product, and the same is true of $\iota(W_t)=T_{k_m}\dotsm T_{k_1}$.
\end{proof}

\subsection{Mallows elements}
We let $\mu_{q,r}(\pi)=r^{-l_0(\pi)}q^{-l_1(\pi)}Z_{q,r}^{-1}$, where (see \cite{M:Coxeter})
\begin{equation*}
    Z_{q,r}=\sum_{\pi\in B_n}r^{-l_0(\pi)}q^{-l_1(\pi)}=\prod_{i=1}^{n} \frac{1-q^{-i}}{1-q}(1+q^{i-1}r).
\end{equation*}
Note that when $q=r$, this reduces to the usual normalizing constant for the Mallows distribution on $B_n$. This is the stationary distribution for the colored ASEP on a finite half open segment, which can be easily verified by checking that the ASEP is reversible with respect to this measure. 
We will let $\mathcal{M}_{[a,b]}$ denote the \textbf{Mallows element} of $\mathcal{H}$ associated to $[a,b]$, defined by
\begin{equation}\label{def:MallowsElement}
\mathcal{M}_{[a,b]}=\sum_{w\in B_{[a,b]}}r^{-l_0(w)}q^{-l_1(w)}Z_{q,r}^{-1}T_w
\end{equation}
if $a=0$, and
\begin{equation}\label{def:MallowsElement2}
\mathcal{M}_{[a,b]}=\sum_{w\in B_{[a,b]}}q^{-l(w)}Z_{q}^{-1}T_w,
\end{equation}
where $Z_{q}=\prod_{i=1}^n\frac{1-q^{-i}}{1-q}$. Here, we are viewing $B_{[a,b]}$ as a subgroup of $B_n$ in the natural way.
\begin{lemma}
For any $\pi\in B_{[a,b]}$, for some $0\leq a < b$, we have
\begin{equation*}
T_\pi\mathcal{M}_{[a,b]}=\mathcal{M}_{[a,b]}T_\pi=\mathcal{M}_{[a,b]}.
\end{equation*}
\end{lemma}
\begin{proof}
It suffices to show this for $\pi=s_k$, since for any reduced word decomposition $\pi=s_{k_1}\dotsm s_{k_l}$, we have $T_\pi=T_{k_1}\dotsm T_{k_l}$. If $k>1$, then we can compute
\begin{equation*}
\begin{split}
    Z_{q,r}T_k\mathcal{M}_{[a,b]}&=\sum_{l(s_kw)>l(w)}r^{-l_0(w)}q^{-l_1(w)}T_{s_k w}+\sum_{l(s_kw)>l(w)}r^{-l_0(s_kw)}q^{-l_1(s_kw)}(qT_{w}+(1-q)T_{s_k w})
    \\&=\sum_{l(s_kw)>l(w)}r^{-l_0(s_k w)}q^{-l_1(s_k w)}T_{s_k w}+\sum_{l(s_kw)>l(w)}r^{-l_0(w)}q^{-l_1(w)}T_{w}
    \\&=Z_{q,r}\mathcal{M}_{[a,b]}.
\end{split}
\end{equation*}
The other computations are extremely similar.
\end{proof}

We note that if $\pi$ is any configuration, then $ \mathcal{M}_{[a,b]}T_\pi$ has the effect of bring the interval $[a,b]$ into equilibrium, which is the Mallows distribution up to a relabeling of the numbers appearing at positions $[a,b]$ while keeping the relative ordering preserved.

\begin{lemma}\label{lem:InvolutionHecke}
We have $\iota(\mathcal{M}_{[a,b]})=\mathcal{M}_{[a,b]}$.
\end{lemma}
\begin{proof}
Note that $l_0(\pi^{-1})=l_0(\pi)$ and $l_1(\pi^{-1})=l_1(\pi)$ as the $s_k$ are involutions, so reversing their order does not change the number of times each $s_k$ occurs. The claim immediately follows since $B_{[a,b]}$ is a subgroup so it is closed under inversion.
\end{proof}

%TODO: need to fix conventions here, also use 4 species system
For $\eta$ obtained as a projection of Mallows element, we obtain precise bounds on the location of the right-most particle $\mathcal{R}(\eta)$. Fix an interval $[a,b]$  and let $\pi \sim \mathcal{M}_{[a,b]}$ on $B_n$ for some $n\in \N$. For $a>0$ and $k \in [b-a]$, let $\eta^{[a,b],k}_{\pi} \in \Omega_{|b-a|}$ be the configuration where we map the the $b-a-k$ largest values in $(\pi(x))_{x \in [a,b]}$ to particles, and the remaining values in $(\pi(x))_{x \in [a,b]}$ to empty sites, and then project to the interval $[a,b]$. For $a=0$ and $k\in [-b,b]$, we obtain $\eta^{b,k}_{\pi} \in \Omega_{b} $ from $(\pi(x))_{x \in [-b,b]}$ by identifying the $b-k$ largest entries with particles, the remaining $k$ entries as empty sites, and then projecting to $[b]$.

Let us point out that the above construction can  be interpreted as a projection of the multi-species extension of the open ASEP with first, second and third class particles. Although not preserving the Markovian dynamics, the locations of holes are preserved. 

\begin{lemma}\label{lem:MallowsBound}  
Let $a>0$ and $k\in [b-a]$. Then for some $c,C>0$, depending on $\alpha,\gamma,q$, 
\begin{equation}
\P( \mathcal{R}(\eta^{[a,b],k}_{\pi}) \geq k+ x ) \leq C \exp( -c x) . 
\end{equation} 
for all $x>0$ large enough. Similarly, when $a=0$ and $k\in [-b,b]$, we get that
\begin{equation}\label{eq:SecondClaimMallows}
\P( \mathcal{R}(\eta^{b,k}_{\pi}) \geq k+ x ) \leq C \exp( -c x) . 
\end{equation} 
In particular, we have that for all $k<0$
\begin{equation}
\P( \eta^{b,k}_{\pi}(x)=1 \text{ for all } x \in [b] ) \geq 1 -  C \exp( -c |k| ) . 
\end{equation}
\end{lemma}
\begin{proof}
For $a>0$, the configuration  $\eta$ has the law of a stationary asymmetric simple exclusion process on a segment of length $b-a$ with $k$ particles. The claim follows from Proposition 4.2 in \cite{BN:CutoffASEP}. For $a=0$, the claim follows from the definition of $\mathcal{M}_{[0,b]}$ in \eqref{def:MallowsElement} and a counting argument for the number of permutations $\pi$ with $\pi(k+i)<0$ for some $i\geq x$.
\end{proof}

\section{Upper bounds for the limit profile}\label{sec:LimitProfileUpperBound}

We adapt the arguments by Bufetov and Nejjar in \cite{BN:CutoffASEP} to convert the estimates on the current in Theorem \ref{thm:CurrentEstimate} into limit profile bounds for the ASEP  with one open boundary. Recall the set of configurations $\mathcal{A}_N$ from \eqref{def:HittingSet} and $F_\rho$ from \eqref{eq:Main}. In the following, our goal is to show that by time $g_{\rho}(c)$, we have left the set $\mathcal{A}_N$ with probability $1-F_{\rho}(c)$. This is formalized in the next proposition.

\begin{proposition}\label{pro:HittingTheSet}
For all $\rho>0$, and for all $\varepsilon>0$,  %Why is there a \varepsilon???
\begin{equation}\label{eq:HittingTheSet}
\liminf_{N \rightarrow \infty}{\P}_{N} \big( \eta_{s} \notin \mathcal{A}_N \text{ for some } s \leq g_{\rho}(c) \, \big| \, \eta_0=\mathbf{0} \big) \geq 1- F_\rho(c) . 
\end{equation}
\end{proposition}

Using Proposition \ref{pro:HittingTheSet}, the upper bound on the limit profile is immediate.
\begin{proof}[ Proof of the upper bound in Theorem \ref{thm:Main}]
We combine Proposition \ref{pro:HittingTheSet}, the upper bound \eqref{eq:UpperBoundHitting}, and the hitting time bound from the extremal configurations in \eqref{eq:HitFromExtreme}. 
\end{proof}

%\subsection{Ideas for the proof of the upper bound}
%
%We start by outlining the argument for the bound in \eqref{eq:HittingTheSet}. We consider two random configurations derived from the ASEP with one open boundary. For the first configuration, we consider the half-space ASEP from an initial configuration which is close to law of the ASEP with one open boundary the segment. This initial configuration by enlarging the segment and bringing parts of it to $Q$-equilibrium, together with the canonical coupling. We compare this to the random configuration obtained by running the ASEP with empty initial condition and afterwards applying the $Q$-equilibrium steps, but in reversed order. The results in Section \ref{sec:ASEPHecke} ensure that we can express the event to see the first process in $\mathcal{A}_N$ as a question about the position of particles in the second process. The latter is studied using the current estimates in Section \ref{sec:CurrentTheorems}, allowing us to conclude. 

\subsection{Preliminaries for the proof of Proposition \ref{pro:HittingTheSet}}
 
To show Proposition \ref{pro:HittingTheSet}, we consider two random configurations $\mathcal{C}_t$ and $\mathcal{D}_t$ derived from the ASEP with one open boundary between time $0$ and $t$. 
We write in the following $\pi(\eta)$ for the signed permutation which yields $\eta \in \Omega_M$ when projecting positive integers to particles and negative numbers to empty sites and has the smallest possible length. For a random permutation $\pi$, we let $\mathcal{H}(\pi)$ denote the corresponding probability distribution as an element in $\mathcal{H}$.
In order to define $\mathcal{C}_t$, fix some $S \in \N$, and define $\mathcal{C}_0 \in \{0,1\}^{S+N}$ as the projection of the random permutation $\pi$ to $B_{N+S}/B_{[1,N+S]}$ (this has the effect of labelling negative numbers as particles and positive numbers as holes) defined via the relation
\begin{equation}
\label{eq: C hecke}
\mathcal{H}(\pi) := \mathcal{M}_{[1,S+N]} \mathcal{M}_{[0,S]} \mathcal{H}(\pi(\mathbf{0})) , 
\end{equation} i.e.\  $\mathcal{C}_0$ is the configuration which we obtain by starting from the empty initial configuration $\mathbf{0}=\{0,\dots\}$, then bringing $[0,S]$ to equilibrium, and then bringing $[1,S+N]$ to equilibrium. Heuristically, this mimics the effect of placing $S$ particles at positions $N+1,\dotsc, N+S$. For the $\mathcal{C}_t$, we consider now the ASEP with one open boundary run at time $t$ when started with the initial configuration $\mathcal{C}_0$. It is the projection of the permutation with distribution $W_t\mathcal{M}_{[1,S+N]}\mathcal{M}_{[0,S]}$.
Similarly, we define a configuration $\mathcal{D}_t \in \{1,2,3,\infty\}^{S+N}$ as follows. We first consider the multi-species extension of the ASEP with one open boundary run at time $t$ and started from the configuration $\xi$ with %WHAT IS 0????
\begin{equation}
\xi(x) = \begin{cases} 3 & \quad \text{if } x \leq \log^{1/16}(N) \\
\infty & \quad \text{if } x > \log^{1/16}(N) . 
\end{cases}
\end{equation} Denote the resulting configuration by $\xi_t$. We then obtain $\mathcal{D}_t$ by first bringing $[1,S+N]$ into equilibrium, and then $[0,N]$ into equilibrium. This has the distribution given by projecting the random permutation obtained by
\begin{equation}
\label{eq: D Hecke}
\mathcal{H}(\pi(\mathcal{D}_t)) = \mathcal{M}_{[0,S]} \mathcal{M}_{[1,S+N]} W_t
\end{equation}
to $B_{N+S}/B_{[1,\log^{1/16}(N)]}\times B_{\log^{1/16}(N)+1,N+S}$ (this has the effect of labeling the numbers in the intervals $[-N-S,-\log^{1/16}(N)-1]$, $[-\log^{1/16}(N),-1]$, $[1,\log^{1/16}(N)]$, and $[\log^{1/16}(N)+1,N+S]$ by labels $1$, $2$, $3$, $\infty$). See Figure \ref{fig:Configruations} for a visualization. 
\begin{figure} \label{fig:Configruations}
\centering
\begin{tikzpicture}[scale=0.6]

 \node (text1) at (-2.5,-2){$\mathcal{C}_0$} ;    

	\node[shape=circle,scale=1.2,draw] (A1) at (0,-2){} ;
	\node[shape=circle,scale=1.2,draw] (A2) at (2,-2) {};
	\node[shape=circle,scale=1.2,draw] (A3) at (4,-2){} ;
	\node[shape=circle,scale=1.2,draw] (A4) at (6,-2) {};
	\node[shape=circle,scale=1.2,draw] (A5) at (8,-2){} ;
	\node[shape=circle,scale=1.2,draw] (A6) at (10,-2) {};
	\node[shape=circle,scale=1.2,draw] (A7) at (12,-2){} ;
	\node[shape=circle,scale=1.2,draw] (A8) at (14,-2) {};
	\node[shape=circle,scale=1.2,draw] (A9) at (16,-2){} ;
	\node[shape=circle,scale=1.2,draw] (A10) at (18,-2) {};
	\node[shape=circle,scale=1.2,draw] (A11) at (20,-2){} ;
	\node[shape=circle,scale=1.2,draw] (A12) at (22,-2) {};

	\node[shape=circle,scale=0.9,fill=red] (CA1) at (22,-2) {};
	\node[shape=circle,scale=0.9,fill=red] (CA2) at (20,-2) {};
	\node[shape=circle,scale=0.9,fill=red] (CA3) at (18,-2) {};
	\node[shape=circle,scale=0.9,fill=red] (CA4) at (16,-2) {};
	\node[shape=circle,scale=0.9,fill=red] (CA5) at (14,-2) {};
	\node[shape=circle,scale=0.9,fill=red] (CA6) at (12,-2) {};
%	\node[shape=circle,scale=0.9,fill=red] (CA7) at (10,-2) {};

	\draw[thick] (A1) to (A2);	
	\draw[thick] (A2) to (A3);	
	\draw[thick] (A3) to (A4);	
	\draw[thick] (A4) to (A5);	
	\draw[thick] (A5) to (A6);	
	\draw[thick] (A6) to (A7);	
	\draw[thick] (A7) to (12.7,-2);
	\draw[thick,dotted] (12.7,-2) to (13.3,-2);		
	\draw[thick] (13.3,-2) to (A8);	
	\draw[thick] (A8) to (A9);	
	\draw[thick] (A9) to (A10);	
	\draw[thick] (A10) to (A11);	
	\draw[thick] (A11) to (A12);		

  \draw [->,line width=1pt] (-1,-1.5) to [bend right,in=135,out=45] (A1);	

    \draw [<-,line width=1pt] (-1,-2.5) to [bend right,in=-135,out=-45] (A1);	
	
 \node (text1) at (-2.5,-4){$\mathcal{C}_t$} ;    

	\node[shape=circle,scale=1.2,draw] (B1) at (0,-4){} ;
	\node[shape=circle,scale=1.2,draw] (B2) at (2,-4) {};
	\node[shape=circle,scale=1.2,draw] (B3) at (4,-4){} ;
	\node[shape=circle,scale=1.2,draw] (B4) at (6,-4) {};
	\node[shape=circle,scale=1.2,draw] (B5) at (8,-4){} ;
	\node[shape=circle,scale=1.2,draw] (B6) at (10,-4) {};
	\node[shape=circle,scale=1.2,draw] (B7) at (12,-4){} ;
	\node[shape=circle,scale=1.2,draw] (B8) at (14,-4) {};
	\node[shape=circle,scale=1.2,draw] (B9) at (16,-4){} ;
	\node[shape=circle,scale=1.2,draw] (B10) at (18,-4) {};
	\node[shape=circle,scale=1.2,draw] (B11) at (20,-4){} ;
	\node[shape=circle,scale=1.2,draw] (B12) at (22,-4) {};

	\node[shape=circle,scale=0.9,fill=red] (CA1) at (22,-4) {};
	\node[shape=circle,scale=0.9,fill=red] (CA2) at (20,-4) {};
	\node[shape=circle,scale=0.9,fill=red] (CA3) at (18,-4) {};
	\node[shape=circle,scale=0.9,fill=red] (CA4) at (16,-4) {};
	\node[shape=circle,scale=0.9,fill=red] (CA5) at (14,-4) {};
	\node[shape=circle,scale=0.9,fill=red] (CA6) at (12,-4) {};
	\node[shape=circle,scale=0.9,fill=red] (CA7) at (10,-4) {};
	\node[shape=circle,scale=0.9,fill=red] (CA8) at (8,-4) {};
	\node[shape=circle,scale=0.9,fill=red] (CA9) at (6,-4) {};
	\node[shape=circle,scale=0.9,fill=red] (CA10) at (4,-4) {};
	\node[shape=circle,scale=0.9,fill=red] (CA11) at (0,-4) {};

	\draw[thick] (B1) to (B2);	
	\draw[thick] (B2) to (B3);	
	\draw[thick] (B3) to (B4);	
	\draw[thick] (B4) to (B5);	
	\draw[thick] (B5) to (B6);	
	\draw[thick] (B6) to (B7);	
	\draw[thick] (B7) to (12.7,-4);
	\draw[thick,dotted] (12.7,-4) to (13.3,-4);		
	\draw[thick] (13.3,-4) to (B8);	
	\draw[thick] (B8) to (B9);	
	\draw[thick] (B9) to (B10);	
	\draw[thick] (B10) to (B11);	
	\draw[thick] (B11) to (B12);		
	
 \node (text1) at (-2.5,-6){$\xi_t$} ;   
 
 \draw [->,line width=1pt] (-1,-3.5) to [bend right,in=135,out=45] (B1);	

    \draw [<-,line width=1pt] (-1,-4.5) to [bend right,in=-135,out=-45] (B1);

	\node[shape=circle,scale=1.2,draw] (C1) at (0,-6){} ;
	\node[shape=circle,scale=1.2,draw] (C2) at (2,-6) {};
	\node[shape=circle,scale=1.2,draw] (C3) at (4,-6){} ;
	\node[shape=circle,scale=1.2,draw] (C4) at (6,-6) {};
	\node[shape=circle,scale=1.2,draw] (C5) at (8,-6){} ;
	\node[shape=circle,scale=1.2,draw] (C6) at (10,-6) {};
	\node[shape=circle,scale=1.2,draw] (C7) at (12,-6){} ;
	\node[shape=circle,scale=1.2,draw] (C8) at (14,-6) {};
	\node[shape=circle,scale=1.2,draw] (C9) at (16,-6){} ;
	\node[shape=circle,scale=1.2,draw] (C10) at (18,-6) {};
	\node[shape=circle,scale=1.2,draw] (C11) at (20,-6){} ;
	\node[shape=circle,scale=1.2,draw] (C12) at (22,-6) {};
	
	\node[shape=star,star points=5,star point ratio=2.5,fill=red,scale=0.45] (YC1) at (0,-6) {};
	\node[shape=star,star points=5,star point ratio=2.5,fill=red,scale=0.45] (YC2) at (2,-6) {};
	\node[shape=circle,scale=0.9,fill=red] (CC) at (4,-6) {};
	\node[shape=circle,scale=0.9,fill=red] (CC) at (6,-6) {};
	\node[shape=circle,scale=0.9,fill=red] (CC) at (10,-6) {};

	\draw[thick] (C1) to (C2);	
	\draw[thick] (C2) to (C3);	
	\draw[thick] (C3) to (C4);	
	\draw[thick] (C4) to (C5);	
	\draw[thick] (C5) to (C6);	
	\draw[thick] (C6) to (C7);	
	\draw[thick] (C7) to (12.7,-6);
	\draw[thick,dotted] (12.7,-6) to (13.3,-6);		
	\draw[thick] (13.3,-6) to (C8);	
	\draw[thick] (C8) to (C9);	
	\draw[thick] (C9) to (C10);	
	\draw[thick] (C10) to (C11);	
	\draw[thick] (C11) to (C12);		

  \draw [->,line width=1pt] (-1,-5.5) to [bend right,in=135,out=45] (C1);	

    \draw [<-,line width=1pt] (-1,-6.5) to [bend right,in=-135,out=-45] (C1);
	
 \node (text1) at (-2.5,-8){$\mathcal{D}_t$} ;    

	\node[shape=circle,scale=1.2,draw] (D1) at (0,-8){} ;
	\node[shape=circle,scale=1.2,draw] (D2) at (2,-8) {};
	\node[shape=circle,scale=1.2,draw] (D3) at (4,-8){} ;
	\node[shape=circle,scale=1.2,draw] (D4) at (6,-8) {};
	\node[shape=circle,scale=1.2,draw] (D5) at (8,-8){} ;
	\node[shape=circle,scale=1.2,draw] (D6) at (10,-8) {};
	\node[shape=circle,scale=1.2,draw] (D7) at (12,-8){} ;
	\node[shape=circle,scale=1.2,draw] (D8) at (14,-8) {};
	\node[shape=circle,scale=1.2,draw] (D9) at (16,-8){} ;
	\node[shape=circle,scale=1.2,draw] (D10) at (18,-8) {};
	\node[shape=circle,scale=1.2,draw] (D11) at (20,-8){} ;
	\node[shape=circle,scale=1.2,draw] (D12) at (22,-8) {};
	
	\node[shape=circle,scale=0.9,fill=red] (CA1) at (22,-8) {};
	\node[shape=circle,scale=0.9,fill=red] (CA2) at (20,-8) {};
	\node[shape=circle,scale=0.9,fill=red] (CA3) at (18,-8) {};
	\node[shape=circle,scale=0.9,fill=red] (CA4) at (16,-8) {};
	\node[shape=circle,scale=0.9,fill=red] (CA5) at (14,-8) {};
	\node[shape=circle,scale=0.9,fill=red] (CA6) at (12,-8) {};
	\node[shape=circle,scale=0.9,fill=red] (CA7) at (10,-8) {};
	\node[shape=circle,scale=0.9,fill=red] (CA8) at (8,-8) {};
	\node[shape=circle,scale=0.9,fill=red] (CA9) at (6,-8) {};
	\node[shape=circle,scale=0.9,fill=red] (CA10) at (4,-8) {};
	\node[shape=star,star points=5,star point ratio=2.5,fill=red,scale=0.45] (YC2) at (2,-8) {};
	\node[shape=circle,scale=0.9,fill=red] (CA11) at (0,-8) {};

	\draw[thick] (D1) to (D2);	
	\draw[thick] (D2) to (D3);	
	\draw[thick] (D3) to (D4);	
	\draw[thick] (D4) to (D5);	
	\draw[thick] (D5) to (D6);	
	\draw[thick] (D6) to (D7);	
	\draw[thick] (D7) to (12.7,-8);
	\draw[thick,dotted] (12.7,-8) to (13.3,-8);		
	\draw[thick] (13.3,-8) to (D8);	
	\draw[thick] (D8) to (D9);	
	\draw[thick] (D9) to (D10);		
	\draw[thick] (D10) to (D11);	
	\draw[thick] (D11) to (D12);		

  \draw [->,line width=1pt] (-1,-7.5) to [bend right,in=135,out=45] (D1);	

    \draw [<-,line width=1pt] (-1,-8.5) to [bend right,in=-135,out=-45] (D1);

 \node (text1) at (2,-9){$\log^{1/16}(N)$} ; 
 \node (text2) at (10,-9){$N$} ; 
 \node (text3) at (14,-9){$S$} ; 
 \node (text4) at (22,-9){$S+N$} ; 
 
%\particles(6.9+4.9+1.6)(0);

%\annhil(0)(0.7);
%\annhil(-13.4)(-0.7);

   %\draw [->,line width=1pt] (6.9+1.6,-0.475) to [bend right,in=135,out=45] (D);	
   %\draw [->,line width=1pt] (D) to [bend right,in=135,out=45] (6.9+1.6,0.475);	

	\end{tikzpicture}	
\caption{Visualization of the different configurations involved in the proof of Proposition~\ref{pro:HittingTheSet} for $t=g^N_{\rho}(c)$. First class particles are marked as red dots, second class and third class particles are depicted as stars. }
 \end{figure}
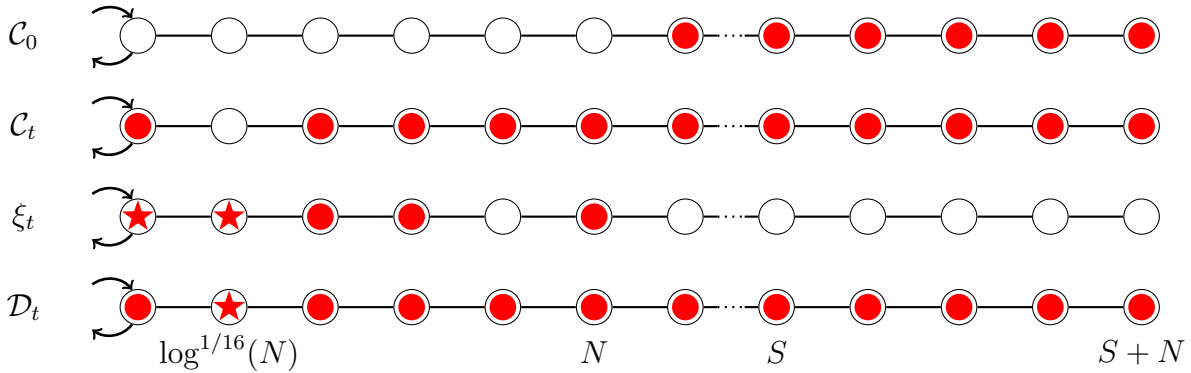
Our key observation is that by the anti-involution property in Lemma \ref{lem:InvolutionHecke} and the projection of Lemma \ref{lem:CouplingMultiSpecies}, we can express the event that $\{\mathcal{C}_t \notin \mathcal{A}_{N} \}$ using the configuration $\mathcal{D}_t$. 
%This is the analogue of Proposition~5.3 in~\cite{BN:CutoffASEP}. 
\begin{lemma}\label{lem:ASEPDuality} There exists a coupling such that for all $t \geq 0$, 
\begin{equation}
\{\mathcal{C}_t \notin \mathcal{A}_{N} \} = \{ \mathcal{D}_t(x) \neq \infty \text{ for all } x\in [S+N] \}
\end{equation}
\end{lemma}
\begin{proof} This follows from the anti-involution $\iota$ and Lemmas \ref{lem: W_t inv} and \ref{lem:InvolutionHecke}, as the distribution of the permutation defined by \eqref{eq: C hecke} and the distribution of the inverse of the permutation defined by \eqref{eq: D Hecke} are the same. The event that $\mathcal{C}_t\notin \mathcal{A}_N$ is the projection of the event that $\pi(j)\leq -1$ for all $j\geq \log^{1/16}(N)$. The inverse of this event is that $\pi^{-1}(j)\leq -1$ for all $j\geq \log^{1/16}(N)$, and the projection of this is exactly that $\mathcal{D}_t(x)\neq \infty$ for all $x\in [S+N]$.
\end{proof}

\subsection{From current estimates to limit profiles}

In the remainder, we have to argue that the configuration $\mathcal{C}_0$ is sufficiently close to the step initial condition of the ASEP with one open boundary, and that the probability of the event that $\mathcal{D}_t$ contains no holes can be expressed using the current on the half space ASEP. Both steps are similar to arguments presented in Section 5.7 of \cite{BN:CutoffASEP}. 
We start by estimating the probability that the configuration $\mathcal{D}_t$ under the above construction contains only first, second, and third class particles for a suitable choice of $t$. Recall the coupling of the ASEP with one open boundary to the half space ASEP, and set  $S=N^{N}$.

\begin{lemma}\label{lem:DualEstimate} Recall the function $g_{\rho}(c)$ from \eqref{def:TimeFunction}. Then for all $c \in \R$, and all $\varepsilon>0$, there exists some $N_0(\varepsilon,c)$ such that for all $S,N \geq N_0$ 
\begin{equation}\label{eq:NoHoles}
{\P}(\mathcal{D}_{g_{\rho}(c)}(x) \neq \infty \text{ for all } x\in [S+N]) \geq  {\P}\big(\mathcal{J}_{g_{\rho}(c)}^{\N} \geq N +3\log^{1/16}(N)\big) - \varepsilon
\end{equation} In particular, with the current estimates in Theorem \ref{thm:CurrentEstimate}, we obtain that
\begin{equation}\label{eq:LowerBoundHoles}
\liminf_{N \rightarrow \infty} {\mathbb{P}}(\mathcal{D}_{g_{\rho}(c)}(x) \neq \infty \text{ for all } x\in [S+N]) \geq 1-F_{\rho}(c) , 
\end{equation}  where we recall the function $F_{\rho}(c)$ from \eqref{eq:Main}.
\end{lemma}
\begin{proof}
Let us first explain heuristically the idea. By choosing $S$ very large, the finite and infinite ASEP agree with high probability. By waiting until time $g_\rho(c)$, at least $N$ first and second class particles will have entered the system in the multispecies process with probability $1-F_\rho(c)$. The effects of bringing the intervals $[1,S+N]$ and then $[0,S]$ into equilibrium essentially have the effect of moving these $N$ first and second class particles to the rightmost $N$ positions, and then moving all holes in $[1,S]$ out of the system. At the end, we clearly have no holes, and so the probability that $\mathcal{D}_t\neq \infty$ for all $x\in [S+N]$ is at least $1-F_\rho(c)$. We now make this precise.

Recall the construction of the multi-species exclusion process $(\xi_{t})_{t \geq 0}$ from Section~\ref{sec:CanonicalCoupling}, started from $\mathcal{D}_0$, and let $(\xi^{\prime}_{t})_{t \geq 0}$ denote the projection in Lemma \ref{lem:CouplingMultiSpecies} to the ASEP with one open boundary. Since $(\xi_{t})_{t \geq 0}$ contains at most $2\log^{1/16}(N)$ second or third class particles, Lemma \ref{lem:CouplingMultiSpecies} and Lemma \ref{lem:CurrentComparison}, and the fact that as $S=N^{N}$ the current of $(\xi^{\prime}_{t})_{t \geq 0}$ agrees under the canonical coupling with probability tending to $1$ with the current of a half space ASEP until time $t=g_{\rho}(c)$, ensure that
\begin{equation}\label{def:EventA1}
{\P}\Big( \big| x \in [S+N] \, \colon \,  \xi_{g_{\rho}(c)}(x) =1 \big| \geq N + \log^{\frac{1}{16}}(N)\Big) \geq  {\P}\big(\mathcal{J}_{g_{\rho}(c)}^{\N} \geq N +3\log^{\frac{1}{16}}(N)\big) - \frac{\varepsilon}{3}
\end{equation} for any fixed $\varepsilon>0$ when $N$ is sufficiently large. 
 Then conditioning on the event at the left-hand side of \eqref{def:EventA1}, note that by the first statement in Lemma \ref{lem:MallowsBound}, the configuration $\mathcal{D}_0^{\prime}$ obtained by projection of the permutation with distribution
 \begin{equation}
 \mathcal{H}(\pi(\mathcal{D}_0^{\prime})) := \mathcal{M}_{[1,S+N]} W_t ,
 \end{equation}
 satisfies
\begin{equation}
\P\left( \mathcal{D}_0^{\prime}(x)=1  \text{ for all } x \in [S,S+N] \right) \geq  {\P}\big(\mathcal{J}_{g_{\rho}(c)}^{\N} \geq N +3\log^{\frac{1}{16}}(N)\big) - \frac{2\varepsilon}{3}
\end{equation} for all $N$ large enough. 
Using the second statement in Lemma \ref{lem:MallowsBound} for $k=-\log^{1/16}(N)$ to bound the probability after applying the Mallows element $\mathcal{M}_{[0,S]}$, we conclude.
\end{proof}

 We have all tools to show Proposition \ref{pro:HittingTheSet}, and thus to finish the proof of Theorem~\ref{thm:Main}. What remains is to show that the behaviors of the rightmost hole in the original Markov chain on $[N]$ and $\mathcal{C}_t$ on $[N+S]$ agree with high probability.

\begin{proof}[Proof of Proposition \ref{pro:HittingTheSet}]
Note that $\mathcal{C}_0\succeq \eta_0$ as at most $S$ particles will be in $[1,N+S]$, and so by Lemma \ref{lem:DominationLeftmostParticle}, $\P(\mathcal{R}(\eta_t)\leq \mathcal{R}(\mathcal{C}_t)\text{ for all }t \geq 0)=1$. The result immediately follows from Lemma \ref{lem:ASEPDuality} and Lemma \ref{lem:DualEstimate}.
\end{proof}

\bibliography{LimitASEP}{}
\bibliographystyle{plain}

\textbf{Acknowledgment.} We thank the \textit{Mathematisches Forschungsinstitut Oberwolfach} for the seminar \textit{The Cutoff Phenomenon for Finite Markov Chains}, where this work was initiated. DS acknowledges the DAAD PRIME program for financial support.

\end{document}